%% file: bifur_obstruction_paper_NZJM_clean.tex
\pdfoutput=1

\documentclass{nzjm}
\usepackage{amsmath,amsthm,amsfonts,amssymb,amsbsy,latexsym,graphicx}
\usepackage[colorlinks=true, linkcolor=blue]{hyperref}

\usepackage{xcolor}
\definecolor{green}{rgb}{0,0.5,0}
\usepackage{cancel}

\def\currentvolume{**}
\def\currentyear{****}
\def\pagerange{**--**}
\def\d{\mathrm{d}}
\def\p{\partial }

\def\D{\mathrm{D}}

\def\e{\epsilon}

\def\Z{\mathbb{Z}}

\def\N{\mathbb{N}}
\def\R{\mathbb{R}}

\newtheorem{theorem}{Theorem}[section]
\newtheorem{corollary}[theorem]{Corollary}
\newtheorem{lemma}[theorem]{Lemma}
\newtheorem{prop}[theorem]{Proposition}

\theoremstyle{definition}
\newtheorem{remark}[theorem]{Remark}
\newtheorem{example}[theorem]{Example}
\newtheorem{definition}[theorem]{Definition}
\newtheorem{observation}{Observation}[section]

\numberwithin{equation}{section}

\begin{document}

\title[Hamiltonian bvps, conformal symplectic symmetries, conjugate loci]{Hamiltonian boundary value problems, conformal symplectic symmetries, and conjugate loci}
\author[R. McLachlan]{Robert I McLachlan}
\author[C. Offen]{Christian Offen}
%\date{20 April, 2018}
%\thanks{We thank Peter Donelan, Bernd Krauskopf, Hinke Osinga and Gemma Mason for useful discussions. This research was supported by the Marsden Fund of the Royal Society Te Ap\={a}rangi.}

\input{abstract}

\maketitle

\input{intro_geodesicbifur}

\input{introduction}

\input{dirichlet_extra_structure}

\input{acknowledgements}

\bibliographystyle{spmpsci}  
\bibliography{resources}

\smallskip
\smallskip
\smallskip
\smallskip
\smallskip
\smallskip
\smallskip

\begin{minipage}{0.31\textwidth}
%you mustn't put a blank-line break in here!
\begin{flushleft}
\begin{footnotesize}
Robert I McLachlan\\ Institute of Fundamental Sciences\\ Massey University\\ Palmerston North \\ New Zealand\\ 
+64 6 951 7652\\
r.mclachlan@massey.ac.nz
\end{footnotesize}
\end{flushleft}
\end{minipage}
\noindent\begin{minipage}{0.31\textwidth}
\begin{flushleft}
\begin{footnotesize}
Christian Offen\\ Institute of Fundamental Sciences\\ Massey University\\ Palmerston North \\ New Zealand\\
+64 6 951 8707\\
c.offen@massey.ac.nz
\end{footnotesize}
\end{flushleft} %you mustn't put a blank-line break in here!
\end{minipage}
%you mustn't put a blank-line break in here!
\end{document}

%% file: abstract.tex
\begin{abstract}
In this paper we continue our study of bifurcations of solutions of boundary-value problems for symplectic maps arising as Hamiltonian diffeomorphisms. These have been shown to be connected to catastrophe theory via generating functions and ordinary and reversal phase space symmetries have been considered. Here we present a convenient, coordinate free framework to analyse separated Lagrangian boundary value problems which include classical Dirichlet, Neumann and Robin boundary value problems. The framework is then used to {prove the existence of obstructions arising from} conformal symplectic symmetries {on the bifurcation behaviour of solutions to Hamiltonian boundary value problems.}
Under non-degeneracy conditions, {a group action by conformal symplectic symmetries has the effect that the flow map cannot degenerate in a direction which is tangential to the action. This imposes restrictions on which singularities can occur in boundary value problems.} {Our results generalise} classical results about conjugate loci on Riemannian manifolds {to a large class of Hamiltonian boundary value problems with, for example, scaling symmetries}.
\end{abstract}

%% file: intro_geodesicbifur.tex
\section{Introduction}\label{sec:introductionsection}

In the most frequently studied situation of a group acting on a symplectic manifold, the group acts by symplectic or Hamiltonian actions and leaves a Hamiltonian flow invariant.
In another case \cite{ConfHamSys}, the group acts by Hamiltonian actions but the flow is conformal symplectic. In contrast, in this paper we consider conformal symplectic actions on Hamiltonian flows and their effects on bifurcations in boundary value problems.

As an example of a Hamiltonian boundary value problem let us consider the conjugate points problem for geodesics on Riemannian manifolds. Recall that geodesics are locally length minimising. Roughly speaking, two points $q$ and $Q$ are conjugate if a geodesic starting at $q$ stops to be length minimising when reaching $Q$. More precisely, two points $q$ and $Q$ on a Riemannian manifold connected by a geodesic $\gamma$ are called \textit{conjugate points} if there exists a non-trivial Jacobi vector field along $\gamma$ vanishing at $q$ and $Q$. In other words, there exists a non-trivial vector field along $\gamma$ which arises as a variational vector field for variations through geodesics fixing $q$ and $Q$. \cite[Ch.5]{doCarmo} 
{ {Alternatively one can define $q$ and $Q$ to be conjugate if $Q$ is the image of a critical point of the geodesic exponential map at $q$}} \cite{WallGenericManifolds}.
The set of all points conjugate to $q$ is called the \textit{conjugate locus to $q$}. 
The left plot of figure \ref{fig:EllipsoidConjPts} shows the graph of a (slightly perturbed) 2-dimensional Gaussian and the conjugate locus to the point $q$ marked as $\ast$ in the plot. There are three geodesics connecting $q$ with a point $Q$ in between the solid black lines where $Q_1>0$ while there is a unique geodesic if $Q$ is outside that region.
Keeping $q$ fixed and varying $Q$ two of the geodesics merge in a fold bifurcation as $Q$ crosses one of the solid black lines. If $Q$ crosses the meeting point of the lines of folds all connecting geodesics merge into one. The meeting point corresponds to a cusp singularity.

Other examples of conjugate points are antipodal points on a sphere. Let us fix the point $q=(q^1,q^2,q^3)=(-\frac 1 \pi, 0,0)$ on a 2-sphere in $\R^3$ of circumference 2. Its only conjugate point is the anti-podal point $Q=(\frac 1 \pi, 0,0)$. Perturbing the sphere to an ellipsoid the conjugate locus consists of four cusps connected by lines of fold singularities as seen in the plot to the right of figure \ref{fig:EllipsoidConjPts}, unless $q$ happens to be an umbilic point of the ellipsoid. This is known as the last geometric statement of Jacobi \cite{Itoh2004}.

\begin{figure}
\begin{center}
\includegraphics[width=0.6\textwidth]{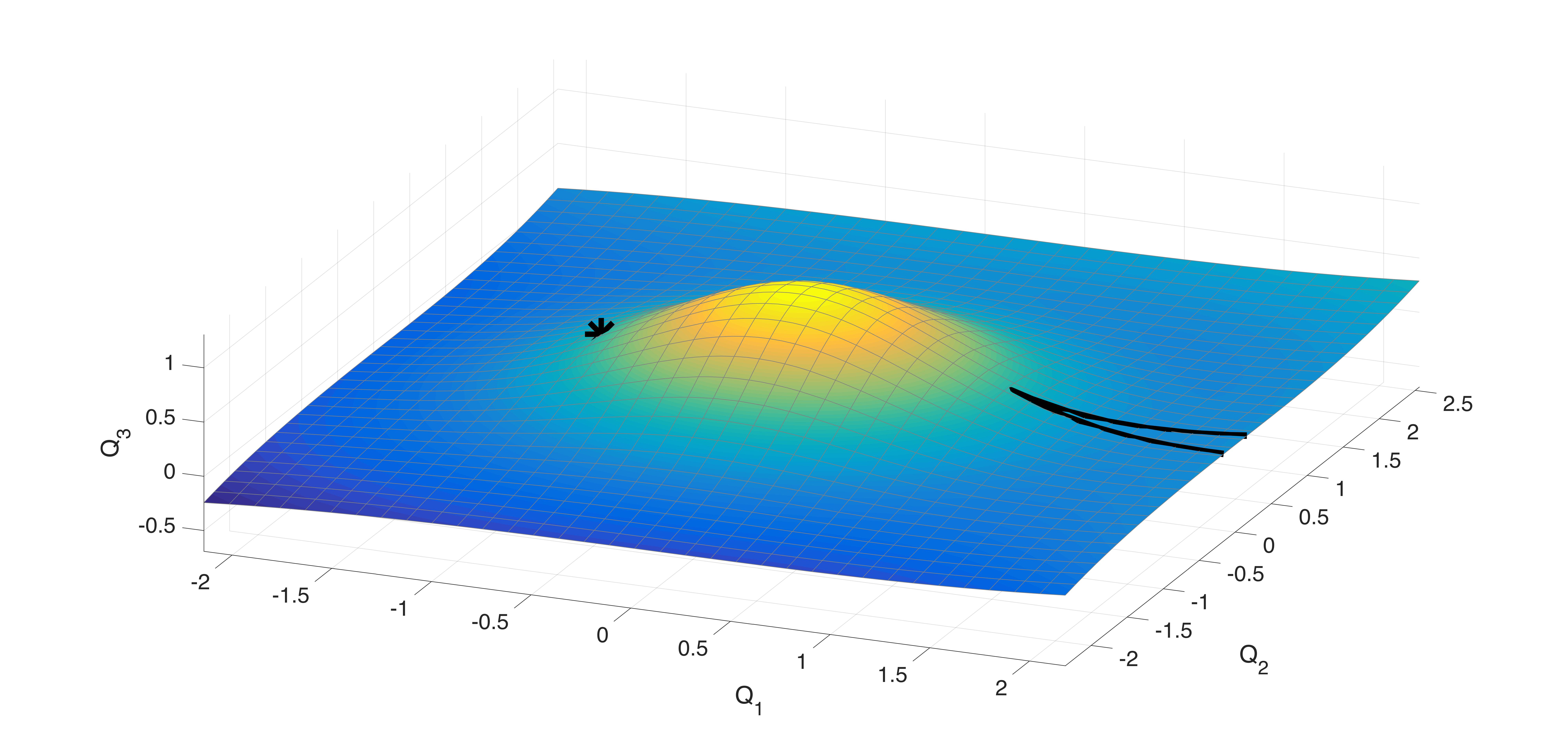}
\includegraphics[width=0.39\textwidth]{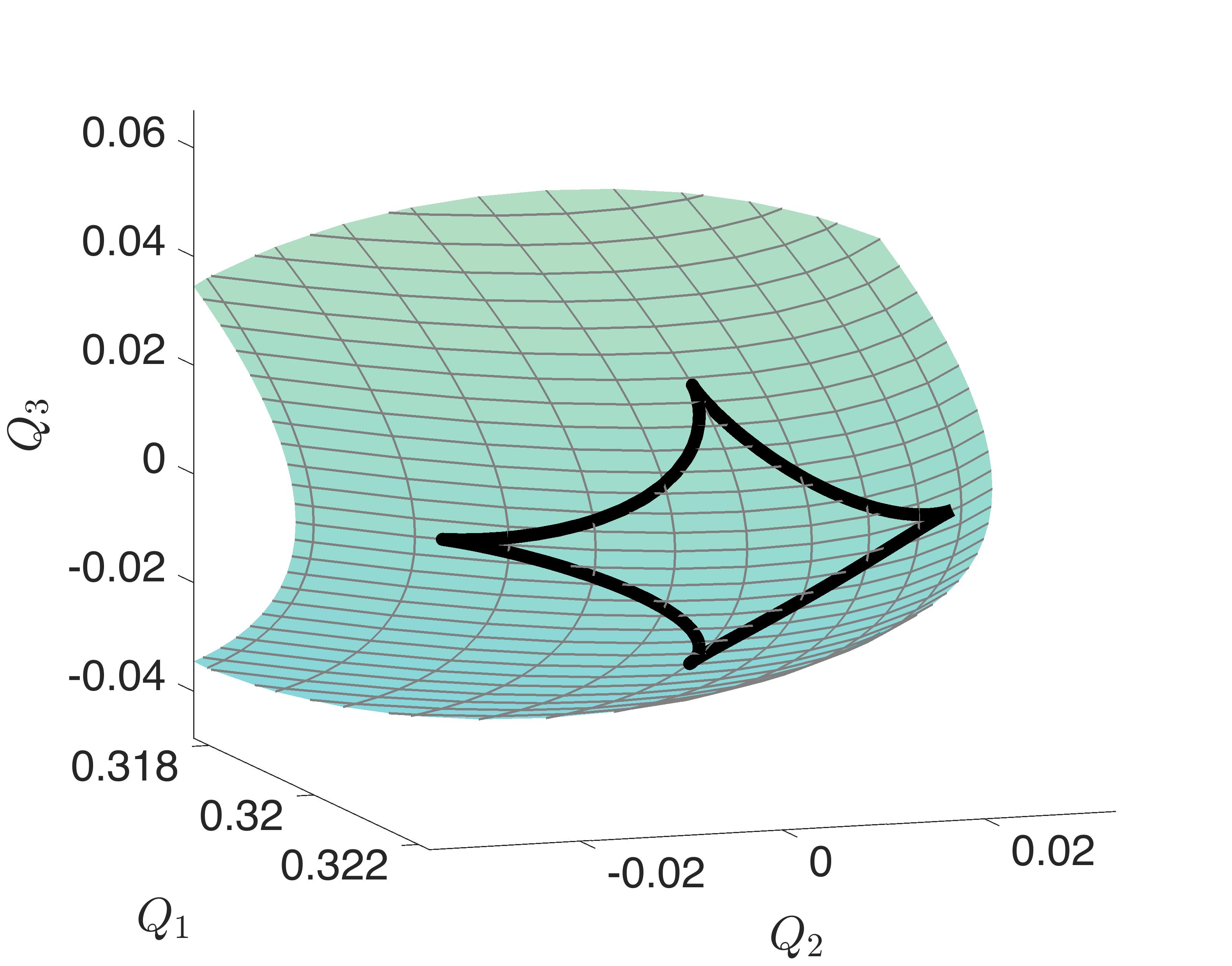}
\end{center}
\caption{Conjugate loci on a perturbed graph of a 2-dimensional Gaussian and on an ellipsoid.}\label{fig:EllipsoidConjPts}
\end{figure}

In general, on a Riemannian manifold $(N,g)$ geodesic motions can be interpreted as motions of a Hamiltonian system defined on the cotangent bundle $\pi \colon T^\ast N \to N$ equipped with the standard symplectic form: for the Hamiltonian
\begin{equation}\label{eq:HamGeoCordsfree}
H(\alpha) = \frac 12 \alpha\left(g^{-1}(\alpha)\right), \quad \alpha \in T^\ast N
\end{equation}
motions correspond to velocity vector fields of geodesics under the bundle isomorphism $g$ between the tangent bundle $TN$ and the cotangent bundle $T^\ast N$. If $\Phi$ is the associated Hamiltonian flow at time 1 then the problem of connecting two points $q,Q \in N$ by a geodesic can be formulated as the following boundary value problem for $\Phi$: find $\alpha \in T^\ast N$ such that
\begin{equation}\label{eq:bdproblHGeodesic}
\pi (\alpha) =q, \quad (\pi \circ \Phi)(\alpha)=Q.
\end{equation}
The Hamiltonian formulation reveals symplectic structure hidden in the geodesic problem: the above problem \eqref{eq:bdproblHGeodesic} is a boundary value problem for a symplectic map.
Indeed, the symplectic structure {has an effect on the kind of bifurcations} which can occur in boundary value problems {and whether the singularities persist under small perturbations}.
Moreover, the Hamiltonian $H$ \eqref{eq:HamGeoCordsfree} has a conformal symplectic symmetry under which $\Phi$ is invariant. In contrast to boundary value problems for arbitrary symplectic maps, the Dirichlet-type form of the boundary condition and the symmetry properties turn out to impose restrictions on which of these singularities actually can occur.

The remainder of the paper is structured as follows.
In section \ref{sec:DirichletExtra} we introduce the notion of \textit{Lagrangian boundary value problems} and recall some results from \cite{bifurHampaper} about their bifurcation behaviour which is connected to catastrophe theory.
A coordinate-free framework is provided to analyse \textit{separated Lagrangian boundary value problems} which include classical Dirichlet, Neumann and Robin boundary conditions and, in particular, our motivational example.

In section \ref{sec:ObstrConfSym} the framework is used to prove the main result of this paper: if a Hamiltonian is invariant under a conformal symplectic action of a $k$-dimensional Lie group then the degree of degeneracy of a singularity in a separated Lagrangian boundary value problem cannot exceed $n-k$, where $2n$ is the phase space dimension, if the action is tangential to the boundary condition and under non-degeneracy conditions on the group action and the Hamiltonian vector field.
This extends the results of \cite[Section 3.2]{bifurHampaper} where the authors prove that separated Lagrangian boundary value problems can only degenerate up to dimension $n$ and that this is the lowest upper bound one can achieve in a general setting.
The new result applies in particular to homogeneous Hamiltonians as \eqref{eq:HamGeoCordsfree}.
This provides an alternative approach to the conjugate-points problem on Riemannian manifolds and recovers classical results about conjugate loci \cite{WallGenericManifolds, CutLocusWeinstein}.

%Let us introduce Lagrangian boundary value problems for symplectic maps, provide more examples of such problems and recall from \cite{bifurHampaper} how the symplectic structure effects the bifurcation behaviour of solutions to such problems.

%The formation divides the surface into an outer part containing $q$ and an inner part which does not contain $q$. 
%Up to periodicity there are four geodesics connecting $q$ with a point $Q$ in the inner part. 
%Moving $Q$ through a black edge from inside the formation to the outer part, two connecting geodesics merge and vanish and the other one persist. Moving $Q$ through a cusp three geodesics merge into one which persists next to the fourth one (the long one).

%% file: introduction.tex
%\section{Introduction}

\section{Lagrangian boundary value problems}\label{sec:DirichletExtra}\label{subsec:introDiri}

\subsection{Definitions and the connection to catastrophe theory}

\begin{definition}[Lagrangian boundary value problem for a symplectic map]\label{def:LagBVP}
Consider a symplectic map $\phi\colon (M,\omega) \to ( M' ,  \omega')$ and projections $\pi \colon M\times  M' \to M$ and $ \pi' \colon M\times  M \to  M'$. Define the symplectic form $\omega \oplus (- \omega'):= \pi^\ast\omega - {\pi'}^\ast \omega'$ on the manifold $M \times  M'$. %A typical example is $(\tilde M , \tilde \omega)=(M,\omega)$ and $\phi$ . 
The graph $\Gamma$ of $\phi$ constitutes a Lagrangian submanifold in the symplectic manifold $(M \times  M', \omega \oplus (- \omega'))$. 
If $\Pi$ is another Lagrangian submanifold in $(M \times  M', \omega \oplus (-\omega'))$, then $(\phi,\Pi)$ is called \textit{Lagrangian boundary value problem (for $\phi$)}. Its solution is the intersection of $\Gamma$ with $\Pi$.
\end{definition}

%In applications symplectic maps often arise as time-$\tau$-maps of Hamiltonian systems on symplectic manifolds $(M,\omega)$ for fixed time $\tau$. In that case $(M,\omega)=(\tilde M, \tilde \omega)$.

\begin{remark}
Lagrangian boundary value problems $(\phi, \Pi)$ with $\Gamma=\mathrm{graph}\, \phi$ can be localized near a solution $z \in \Pi\cap \Gamma$: shrink $M$ to an {open} neighbourhood of $\pi(z)$, $ M'$ to an {open} neighbourhood of $ \pi'(z)$ and restrict $\omega$, $ \omega'$ and $\phi$ accordingly.
\end{remark}

\begin{example} If $(M,\omega)=( M',  \omega')$ than the periodic boundary value problem $\phi(z) = z$ is a Lagrangian boundary value problem where $\Pi =\{(m,m)\, |\, m \in M\}$ is the diagonal.
\end{example}

%\textit{Example.} Consider an open subset $M$ of $\R^{2n}$ with the standard symplectic form $\omega=\sum_{j=1}^n \d x^j \wedge \d y_j$ and a symplectic map $(x,y)\mapsto\phi(x,y)$ on $M$ with $x=(x^1,\ldots,x^n)$-component $\phi^X = x\circ \phi$. Fix $x^\ast,X^\ast \in \R^n$. The boundary value problem
%\begin{equation}\label{eq:Dirichletforphi}
%\phi^X(x^\ast,y) = X^\ast
%\end{equation}
%is a Lagrangian boundary value problem.

\begin{example}\label{ex:DNR} Classical Dirichlet, Neumann and Robin boundary value problems for second order ordinary differential equations of the form $\ddot u(t) = \nabla_u G(t,u(t))$ can be regarded as Lagrangian boundary value problems 
{
for the time $\tau=t_1-t_0$-flow map of the Hamiltonian system $H(t,x,y)=\frac 12 \|y\|^2 - G(t,x)$ as illustrated in figure }\ref{fig:DNRFish}{: a motion solves }

\begin{itemize}
\item
{the Dirichlet problem $u(t_0)=x^\ast$, $u(t_1)=X^\ast$ if it starts on the line $x^\ast\times \R$ at time $t_0$ and ends on $X^\ast\times \R$ at time $t_1$,}
\item
{the Neumann problem $\dot u(t_0)=y^\ast$, $\dot u(t_1)=Y^\ast$ if it starts on $\R\times y^\ast$ and ends on $\R\times Y^\ast$ or }
\item
{the Robin problem $u^j(t_0)+\alpha_0^j \dot u^j(t_0) = \beta_0^j$, $u^j(t_1)+ \alpha_1^j \dot u^j(t_1) = \beta_1^j$ if it starts on the start-line and ends on the end-line plotted in the phase portrait.}
\end{itemize}
{
We have
$\Pi=\{x^\ast\}\times \R \times \{X^\ast\} \times \R$ for the Dirichlet problem,
$\Pi=\R\times \{y^\ast\} \times \R \times \{Y^\ast\} \times \R$ for the Neumann problem
and $\Pi=\{(x,y,X,Y)\, | \, x^j+\alpha_0^j y^j = \beta_0^j,\, X^j+ \alpha_1^j Y^j = \beta_1^j\}$ for the Robin problem.
}
\end{example}

%Indeed, the classical Dirichlet problem
%\begin{equation}\label{eq:DiribdcondODE}
%u(t_0)=x^\ast, \quad u(t_1)=X^\ast
%\end{equation}
%with $t_0<t_1 \in \R$, $x^\ast,X^\ast \in \R^n$ for the second order ordinary differential equation
%\begin{equation}\label{eq:2ndorderODE}
%\ddot u(t) = \nabla_u G(t,u(t))
%\end{equation}
%with a scalar-valued map $G$ defined on a sufficiently large subdomain of $\R \times \R^n$ can be formulated as in \eqref{eq:Dirichletforphi}: let us rewrite \eqref{eq:2ndorderODE} as the first order problem
%\begin{align}\label{eq:1storderODE}
%\dot u(t) &= v(t) \\ \nonumber
%\dot v(t) &=\nabla_u G (t,u(t))
%\end{align}
%and denote by $\phi$ the map which sends a point $(x,y)$ taken from a subdomain of $\R^{2n}$ to the solution of the initial value problem \eqref{eq:1storderODE} with initial data
%\[
%u(t_0)=x, \quad v(t_0)=y
%\]
%(assuming existence and uniqueness). The map $\phi$ is symplectic because it arises as the Hamiltonian flow of
%\begin{equation}\label{eq:mechanicalHam}
%H(t,x,y)=\frac 12 \|y\|^2 - G(t,x).
%\end{equation}
%Now the problem \eqref{eq:2ndorderODE} with Dirichlet boundary conditions \eqref{eq:DiribdcondODE} corresponds to \eqref{eq:Dirichletforphi}. The boundary value problem can be visualised as in the plot to the left of figure \ref{fig:DirichletFish}.

\begin{figure}
\begin{center}
\includegraphics[width=0.32\textwidth]{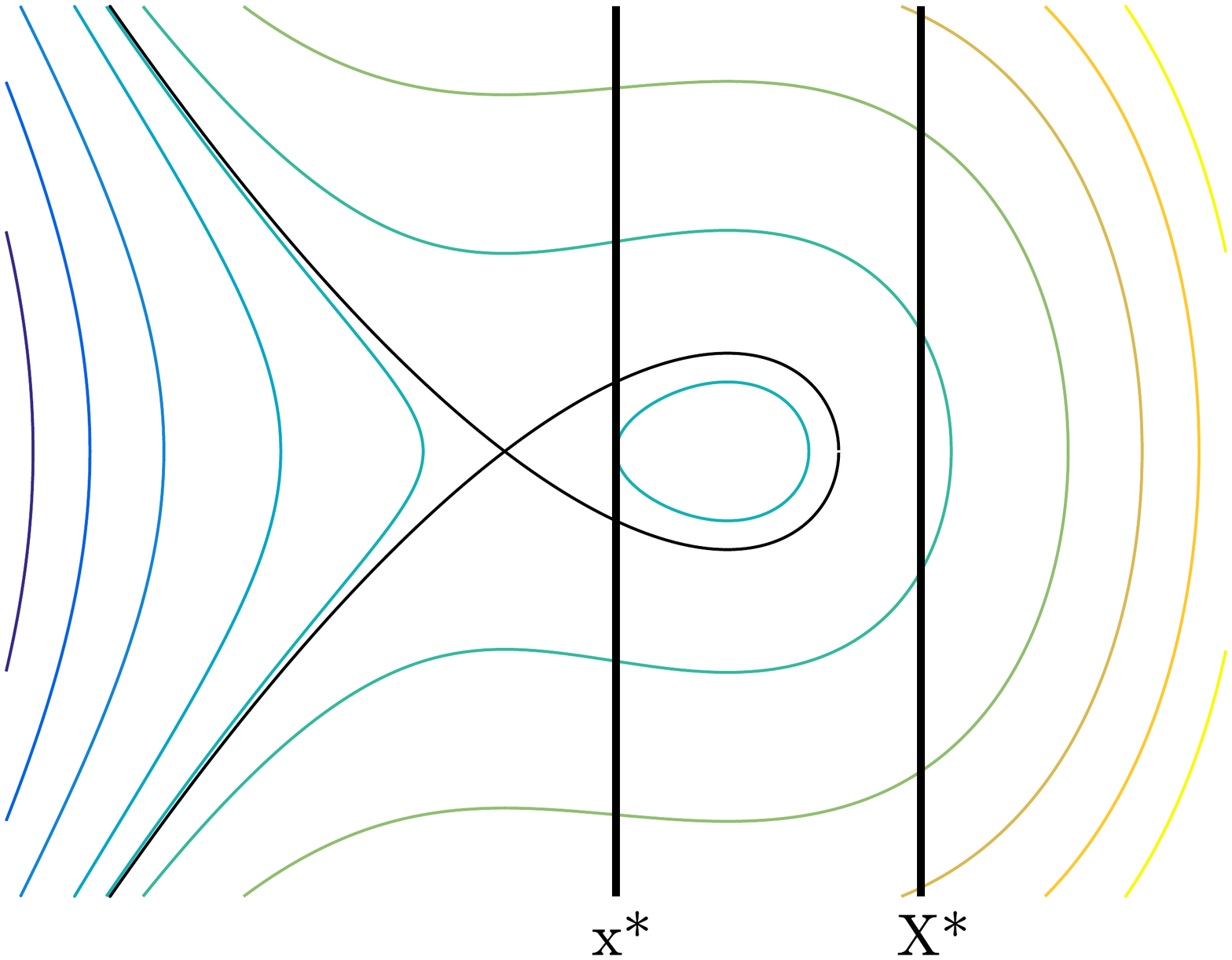}
\includegraphics[width=0.32\textwidth]{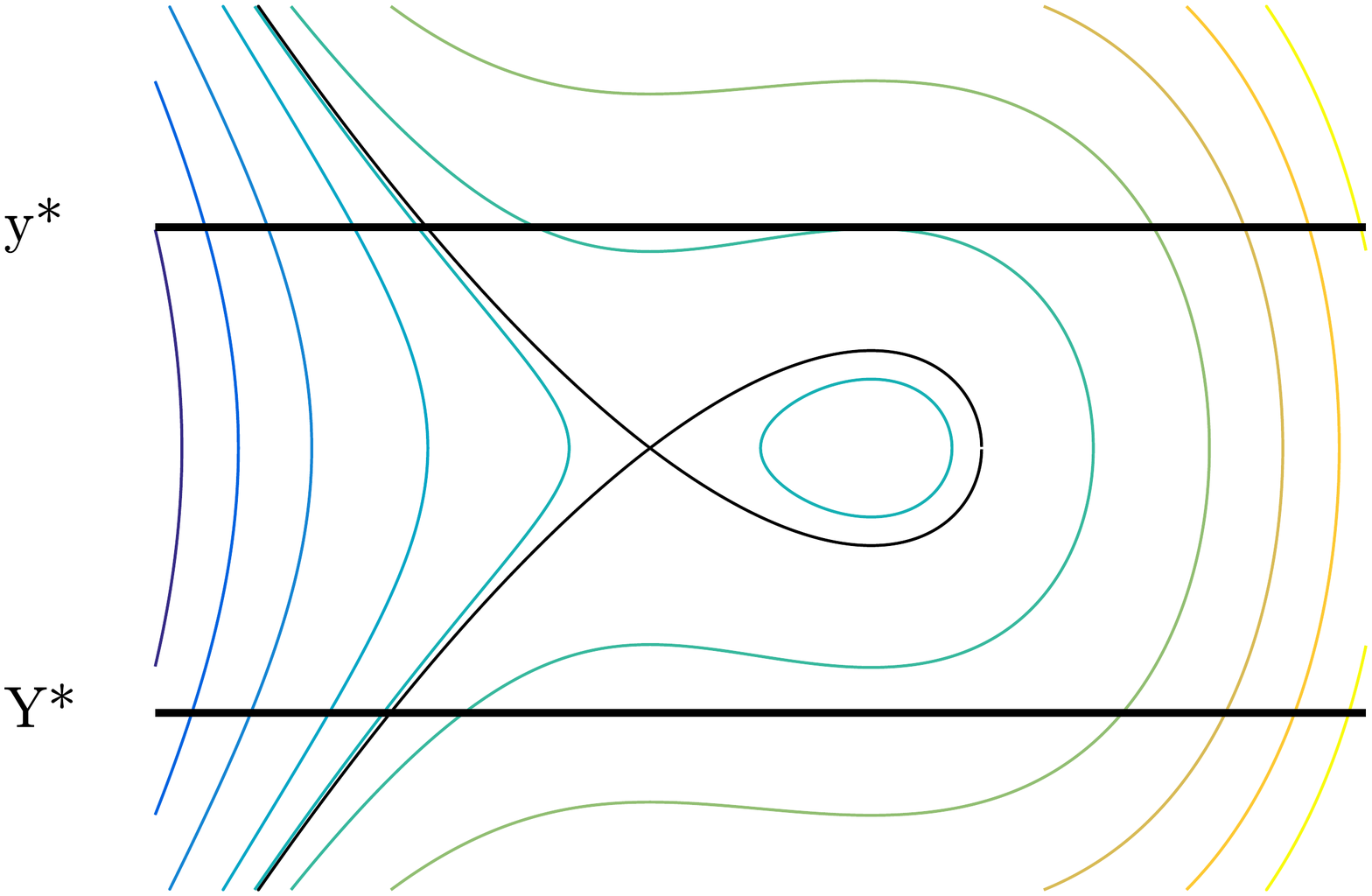}
\includegraphics[width=0.32\textwidth]{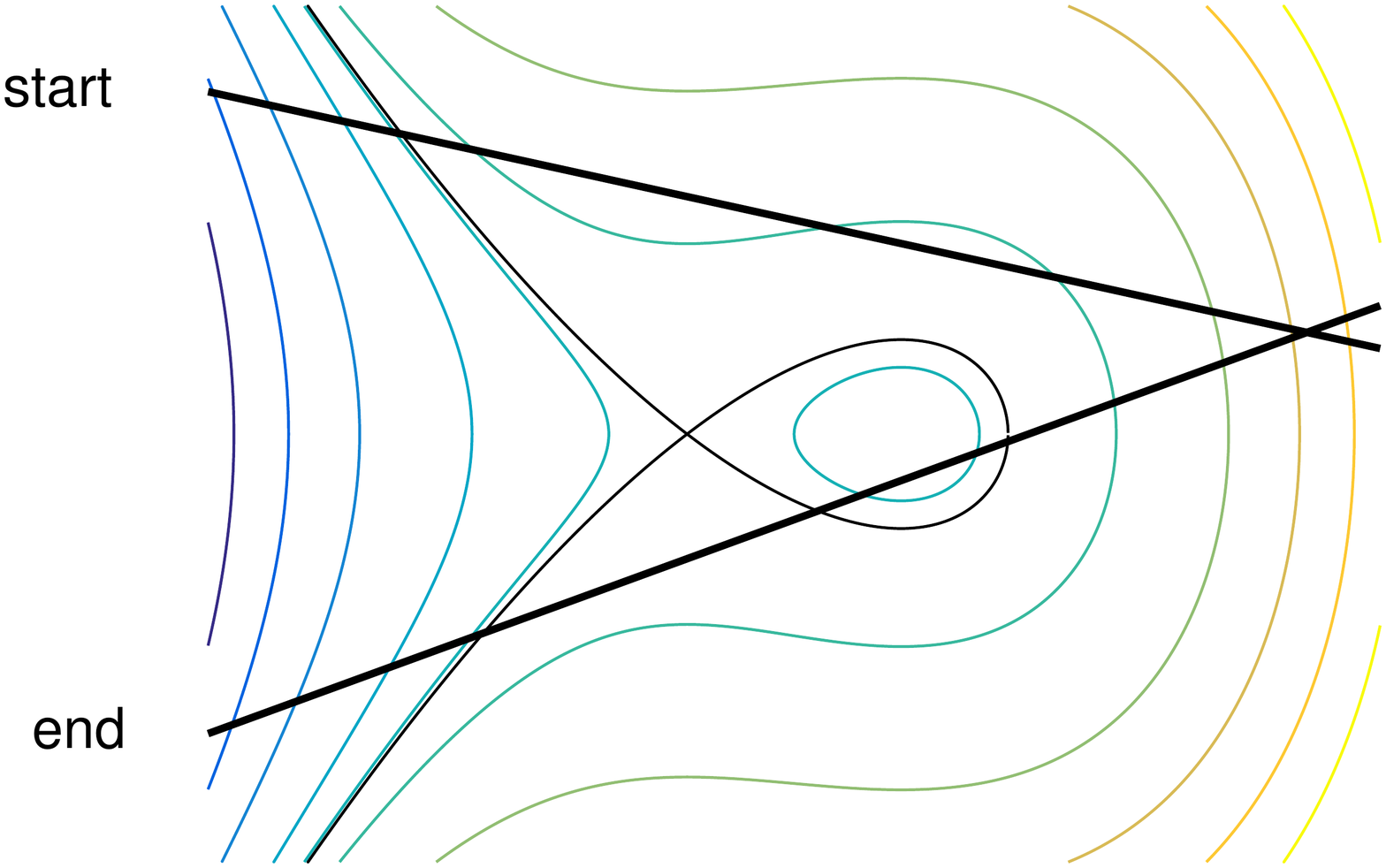}
\end{center}
\caption{Illustration of Dirichlet, Neumann and Robin boundary conditions {as Lagrangian boundary value problems. (Example \ref{ex:DNR})}}\label{fig:DNRFish}
\end{figure}

\begin{example}[Non-example] Let $\omega=\sum_{j=1}^n \d x^j \wedge \d y_j$ be the standard symplectic form on $\R^{2n}$. Consider a symplectic map $\phi \colon \R^{2n}\to \R^{2n}$. The boundary condition $\phi(x,y)=(y,x)$ does \textit{not} yield a Lagrangian boundary value problem
since $\Pi = \{(x,y,y,x)\, | \, (x,y) \in \R^{2n}\}$ is not Lagrangian in $(\R^{2n} \times \R^{2n}, \omega \oplus -\omega)$.
\end{example}

%\begin{remark}
%In the non-example above the boundary value problem is expressed in such coordinates that it is convenient to check whether $\Pi$ is Lagrangian or not. We can, however, conclude that for no description of the boundary value problem as an intersection problem of $\Pi$ and the graph $\Gamma$ definition \ref{def:LagBVP} can be fulfilled.
%%Notice that in the non-example above we can actually conclude from the fact that $\Pi$ is not Lagrangian that definition \ref{def:LagBVP} can not be fulfilled no matter how and in which coordinates we express the boundary value problem.
%In general, whether a boundary value problem is Lagrangian does not depend on its representation as an intersection problem as long as the embedding $\Gamma$ of $M$ into $(M\times \tilde M, \omega \oplus (-\tilde \omega))$ is a Lagrangian embedding.
%\end{remark}
%%Notice that the notion of Lagrangian boundary value problem is invariant under symplectic changes of variables.

\begin{observation}[Bifurcation behaviour of Lagrangian boundary value problems]\label{obs:transl}
The intersection of two Lagrangian submanifolds $\Gamma$ and $\Pi$ arising in Lagrangian boundary value problems can locally be viewed as the intersection of two graphical Lagrangian submanifolds $\tilde \Gamma$ and $\tilde \Pi$ in a cotangent bundle over a simply connected manifold $N$. The Lagrangian submanifolds $\tilde \Gamma$ and $\tilde \Pi$ arise as images of exact 1-forms $\d f$ and $\d h$ on $N$ such that the intersection problem is equivalent to finding points in $N$ where the 1-form $\d g = \d (f-h)$ vanishes. In this way one can assign local critical-points-of-a-function problems to local Lagrangian boundary value problems (see \cite{bifurHampaper} for details). 
\end{observation}

\begin{remark}\label{rem:recapCatastropheTheory}
Parameter-dependent changes of critical points in smooth families of smooth maps
\begin{equation}\label{eq:smoothfamily}
g_\mu\colon \R^k \to \R, \quad z \mapsto g_\mu(z), \quad k \in \N \cup\{0\},
\end{equation}
are called \textit{critical-points-of-a-function} problems or \textit{gradient-zero} problems. They have been treated under the headline \textit{catastrophe theory}. {The parameter $\mu$ can be multi-dimensional.} Roughly speaking, two families localised around a point of interest $z$ and a parameter value $\mu$ are (stably right-left) equivalent if they coincide up to {addition of quadratic forms,} parameter-dependent smooth changes of coordinates on the domain and target space and reparametrisations. Equivalent families show qualitatively the same bifurcation behaviour near $z$ and $\mu$. Generic, local bifurcation phenomena have been described and classified \cite{Arnold1,Gilmore1993catastrophe,lu1976singularity}.
\end{remark}

\begin{definition}[Fully reduced families, degeneracy of singularities] Consider a smooth family $g_\mu$ of smooth maps as in \eqref{eq:smoothfamily} localised at a point $z$ and a parameter value $\mu$. The dimension of the kernel of the Hessian matrix of $g_\mu$ at $z$ is called the \textit{degeneracy of the singularity}. If the number of variables $k$ coincides with the degeneracy of the singularity then the family is called \textit{fully reduced}.
\end{definition}

\begin{remark}
The case $k=0$ in \eqref{eq:smoothfamily} corresponds to non-singular points. This means a fully reduced family without a singular point is a family of nullary functions. Moreover, the count of parameters can be zero such that the notions mentioned above can also be applied to maps.
\end{remark}

%% file: dirichlet_extra_structure.tex
 {We can now transfer the classification of singularities from the catastrophe theory setting to Lagrangian boundary value problems.}

\begin{definition}[Singularity of Lagrangian boundary value problems]\label{def:SingBVP}
{A Lagrangian boundary value problem has a singularity of type $X$ if the generating function $g$ obtained via observation \ref{obs:transl} has a singularity of type $X$ in the sense of catastrophe theory using right (or right-left) equivalence.}
\end{definition}

\begin{remark}\label{rem:REquivalenceJustification}
{We refer to \cite{golubitsky1975contact} for a justification why right-equivalence of generating functions of Lagrangian manifolds corresponds to a sensible geometric notion of contact equivalence for Lagrangian intersection problems. Using the coarser notion of right-left equivalence is sufficient for the viewpoint of this paper such that we do not need to specify either the graph or the boundary manifold as a reference manifold to describe the type of contact.}
\end{remark}

%{The notion of versality described in remark \ref{rem:CatTheoryLagBVP} does not immediately apply if we fix the type of boundary condition and only perturb the symplectic map. Indeed, }
Some boundary conditions impose restrictions on which singularities can occur. 
Given the significance of Dirichlet, Neumann and Robin boundary value problems in applications, let us analyse the bifurcation behaviour in a problem class which we refer to as \textit{separated Lagrangian boundary value problems}.

\subsection{Definition and observations for separated Lagrangian boundary value problems}

% as we will introduce them in definition \ref{def:LagDirproblem}. We recall results from \cite{bifurHampaper} and analyse structure that is present in the data of the problem which can help to locate bifurcation points numerically. As an example, we calculate a $D$-series bifurcation in a H{\'e}non-Heils-type system.

%We develop a coordinate-free description and use the framework to prove obstructions for the degeneracy of singularities in separated Lagrangian boundary value problems in systems with Hamiltonians invariant under a 1-parameter family of conformal symplectic perturbations of the identity. This applies to homogeneous Hamiltonians, in particular, and provides an alternative approach to understand bifurcations of geodesics on Riemannian manifolds. 

%\subsection{Introduction and set-up}\label{subsec:introDiri}

%Let us provide a coordinate free description of separated Lagrangian boundary value problems and formulate some observations which will be helpful to analyse how separated Lagrangian boundary value problems can degenerate, especially for the formulation and proof of theorem \ref{thm:confsymHam}. 

\begin{definition}[Separated (Lagrangian) boundary value problem]\label{def:LagDirproblem}
Let $(M,\omega)$ and $(M',\omega')$ be two symplectic manifolds with $\dim M = 2n = \dim M'$. Consider a symplectic map $\phi \colon M \to M'$ and $n$-dimensional submanifolds $\Lambda \subset M$ and $\Lambda' \subset M'$. The collection $(\phi, \Lambda,\Lambda')$ is called a \textit{separated boundary value problem}. Its solution is given as
\[
\{ z \in \Lambda \, | \, \phi(z) \in \Lambda'\} = \phi^{-1}(\Lambda') \cap \Lambda.
\]
If $\Lambda \subset M$ and $\Lambda' \subset M'$ are Lagrangian submanifolds then $(\phi, \Lambda,\Lambda')$ is called a \textit{separated Lagrangian boundary value problem}.
\end{definition}

\begin{remark} Let $(M,\omega)$ and $(M',\omega')$ be $2n$-dimensional symplectic manifolds and $\Lambda \subset M$ and $\Lambda' \subset M'$ be $n$-dimensional submanifolds. The submanifold $\Pi=\Lambda \times \Lambda'  \subset (M \times M', \omega \oplus -\omega')$ is Lagrangian if and only if $\Lambda$ and $\Lambda'$ are Lagrangian submanifolds. Therefore, the separated boundary value problems which are Lagrangian boundary value problems (definition \ref{def:LagBVP}) are exactly the separated Lagrangian boundary value problems.
\end{remark}

%As we consider singularities and bifurcations in families of Lagrangian Dirichlet problems it is worth noting that 

\begin{example}
The classical Dirichlet, Neumann and Robin boundary value problems considered in example \ref{ex:DNR} constitute separated Lagrangian boundary value problems if regarded as boundary value problems for a Hamiltonian flow map (figure \ref{fig:DNRFish}).
Periodic boundary conditions, however, constitute Lagrangian boundary value problems which cannot be regarded as separated boundary value problems. This leads to a different bifurcation behaviour as we will see in corollary \ref{cor:dimrestr} and proposition \ref{prop:nobifurrestrperiod}.
\end{example}

%Let us formulate some observations about separated Lagrangian boundary value problems.

% Only this local data of families of problems will determine bifurcation behaviour under consideration  \\

\begin{observation}[Local coordinate description]\label{obs:localcoordsDirichlet}
All separated Lagrangian boundary value problems are locally equivalent:
by Darboux-Weinstein's theorem neighbourhoods of Lagrangian submanifolds are locally symplectomorphic to neighbourhoods of the zero section of the cotangent bundle over the submanifolds \cite[Corollary 6.2]{WEINSTEIN1971329}. Therefore, a separated Lagrangian boundary value problem $(\phi, \Lambda,\Lambda')$ for $\phi\colon (M,\omega) \to (M',\omega')$ is locally given as
\begin{equation}\label{eq:DirichletProblem}
x=x^\ast, \quad  \phi^X (x,y)=X^\ast,
\end{equation}
with local Darboux coordinates $(x,y)=(x^1,\ldots x^n, y_1,\ldots y_n)$ for $M$, $(X,Y)=(X^1,\ldots X^n, Y_1,\ldots Y_n)$ for $M'$ and $x^\ast, X^\ast \in \R^{2n}$. In \eqref{eq:DirichletProblem} the symbol $\phi^X$ is a shortcut for $X \circ \phi$. 
%We conclude that the notion of separated Lagrangian boundary value problems in the sense of definition \ref{def:LagDirproblem} extends the classical formulation of Dirichlet problems in dynamical systems and is equivalent in a local setting.\\
In particular, Dirichlet, Neumann and Robin boundary conditions can be treated on the same footing in the bifurcation context. %In contrast, periodic boundary conditions are not separated. This leads to a different bifurcation behaviour (proposition \ref{prop:dimrestr} and \ref{prop:nobifurrestrperiod}).
\end{observation}

Let us consider the separated Lagrangian boundary value problem \eqref{eq:DirichletProblem}
on the phase space $M=M'=\R^{2n}$ with the standard symplectic form $\sum_{j=1}^n \d x^j \wedge \d y_j$.
Introducing a multi-dimensional parameter $\mu$ in the map $\phi$ or in the boundary condition, the bifurcation diagram of  \eqref{eq:DirichletProblem} can be viewed as
\begin{equation}\label{eq:DirichletBifurSet}
\{(\mu,y) \,  | \, h_\mu(y) = 0 \} 
\end{equation}
with
\begin{equation}\label{eq:mapDirichlet}
h_\mu(y) = \phi_{(\mu)}^X(x^\ast,y)-X_\mu^\ast.
\end{equation}

%The bifurcation behaviour of the boundary value problem \eqref{eq:DirichletProblem} is governed by singularity theory which in turn is equivalent to a gradient-zero problem \cite[Section 3]{bifurHampaper}. Naively, without considering extra structure in \eqref{eq:mapDirichlet}, one might mistakenly expect a generic roots-of-a-function type behaviour because symplecticity of the Jacobian matrix $\D \phi(x,y)$ does not force any extra structure on the sub-matrix $D_y \phi^X(x,y)$ at points $(x,y)$ in the phase space.
%However, those small perturbations $\tilde h(y) = h_\mu(y)+\xi_\mu(y)$ of $h_\mu$ which are required to break gradient-zero bifurcations leading to a roots-of-a-function-type behaviour do \textit{not} come from \textit{small} perturbations of $\phi$ through symplectic maps.
%% (and $X^\ast_\mu$ is required to be $y$-independent).
%In other words, unfoldings $h_\mu$ of $h_{\mu^\ast}$ obtained using \eqref{eq:mapDirichlet} for any fixed parameter $\mu^\ast$ are generally not universal in the roots-of-a-function problem.
%Therefore, formulating the problem \eqref{eq:DirichletProblem} in the form \eqref{eq:DirichletBifurSet} does \textit{not} give rise to a contradiction to the correspondence to a gradient-zero problem. Indeed we have

\begin{prop}\label{prop:kernelDIM}
If the dimension of the kernel of the Jacobian matrix of the map \eqref{eq:mapDirichlet} at a parameter value $\mu$ and a value $y$ is $m$ then the degeneracy of the singularity of the corresponding critical-points-of-a-function problem at  $(\mu,y)$ is $m$.\cite[Prop. 6]{numericalPaper}
\end{prop} 

%\textbf{Proof.} 
%
%For a given parameter $\mu$ solutions to the problem \eqref{eq:DirichletProblem} correspond to points in the intersection of the graph $\Gamma_\mu = \{(x,y,\phi^X_\mu(x,y),\phi^Y_\mu(x,y)) \, | \, (x,y) \in \R^{2n}\}$ and  $\Lambda = \{(x^\ast,y,X^\ast,Y)\, | \, y,Y \in \R^n\}$. In the frame $\frac{\p}{\p x^1},\ldots,\frac{\p}{\p x^n},\frac{\p}{\p y_1},\ldots,\frac{\p}{\p y_n}$ the tangent spaces to $\Gamma_\mu$ are spanned by the columns of the matrix
%
%\begin{equation*}%\label{eq:Jacobiiota}
%\begin{pmatrix}
%\Id_{ n}&0\\ 
%0 & \Id_{n} \\ 
%D_x \phi^X_\mu & \D_y \phi^X_\mu\\
%D_x \phi^Y_\mu & \D_y \phi^Y_\mu\\
%\end{pmatrix}.
%\end{equation*}
%Here $\Id_n$ denotes an $n$-dimensional identity matrix. The tangent spaces to $\Lambda_\mu$ are spanned by the columns of the matrix \[
%\begin{pmatrix}
%0&0\\
%\Id_{n}&0\\
%0&0\\
%0&\Id_{n }
%\end{pmatrix}.
%\]
%Since the Jacobian matrix of $h_\mu$ coincides with $\D_y \phi^X_\mu$, the dimension of the kernel of $\D h_\mu$ determines the dimension of the intersection of the tangent spaces at solutions in $\Gamma_\mu \cap \Lambda$.\qed\\
%
%We can formulate the following
%
%\begin{corollary}\label{cor:DseriesifJac0}
%If $n=2$ then the Jacobian matrix $\D_y \phi^X_\mu$ vanishes at a $D$-series bifurcation.
%\end{corollary}

\subsection{Coordinate free framework for separated Lagrangian boundary value problems and obstructions in their bifurcation behaviour}
\quad\\

For further analysis it will be handy to have a coordinate free version of observation \ref{obs:localcoordsDirichlet} and proposition \ref{prop:kernelDIM} available.

Given $\phi\colon (M,\omega) \to (M',\omega')$, consider a separated Lagrangian boundary value problem $( \phi, \Lambda,\Lambda')$.
Localising the problem, if necessary, there is an integrable distribution $\mathcal D$ on $M'$ such that $\Lambda'$ is a leaf of $\mathcal D$ and the projection $M' \to M'/\mathcal D$, $z' \mapsto [z']$ is a submersion.\footnote{\label{footnote:noindstr}An example where $M'/\mathcal D$ does \textit{not} inherit a smooth manifold structure via the projection map is the foliation of a torus $M'=\R^2/\Z$ by dense orbits $t \mapsto (\alpha,1) t$ for irrational $\alpha$. Shrinking $M'$ resolves the problem.}
%The space of leaves $M'/\mathcal D$ inherits the structure of a smooth manifold.
Consider a composition of the restricted map $\phi_{|\Lambda}$ with the projection to the leaf space $M'/\mathcal D$, i.e.\
\[
\left[\phi_{|\Lambda} \right] \colon \Lambda \to M'/\mathcal D.
\]
The solution of the separated Lagrangian boundary value problem corresponds to the preimage of $[\Lambda'] \in M'/\mathcal D$ under the above map.

\begin{prop}\label{prop:kernelDegeneracy}
In the considered setting, if $z \in \Lambda$ is a solution, i.e.\ $\phi(z)\in \Lambda'$, then the dimension of the kernel of the map
\begin{equation}\label{eq:Dirichletnocoords}
\d \left.\left[\phi_{|\Lambda} \right]\right|_z \colon T_z \Lambda \to T_{[z']} \left(M'/\mathcal D\right)
\end{equation}
coincides with the degeneracy of the singularity of the corresponding critical-points-of-a-function problem. In particular,
if the problem occurs in a generic smooth family of separated Lagrangian boundary value problems at a certain parameter value then $z$ is a bifurcation point if and only if the kernel is non-trivial.
\end{prop}

\begin{proof}
The assertion corresponds to proposition \ref{prop:kernelDIM} which is proved in \cite{numericalPaper}.
\end{proof}

\begin{remark} For determining the kernel of \eqref{eq:Dirichletnocoords} we can weaken the condition that $\Lambda'$ is a leaf of $\mathcal D$ to the requirement that $\mathcal D_{z'} = T_{z'}\Lambda'$.
\end{remark}

% The dimension of the kernel corresponds to the minimal amount of variables that is needed to write down the singularity of the corresponding gradient-zero problem (see observation \ref{obs:transl}).
%A fact which will be proven in local coordinates in proposition \ref{prop:kernelDIM}.

\begin{corollary}\label{cor:dimrestr}
In separated Lagrangian boundary value problems the degeneracy of a singularity cannot exceed $\dim \Lambda$, i.e.\ half the phase space dimension.
\end{corollary}
%This recovers \cite[Prop. 3.3]{bifurHampaper}. 

\begin{corollary}
In separated Lagrangian boundary value problems for generic families of symplectic maps on $T^\ast \R$ only $A$-series singularities\footnote{See the ADE classification of singularities, e.g.\ in \cite[p.33]{Arnold1}.} occur. 
\end{corollary}
%\cite[Cor. 3.1]{bifurHampaper}

Notice that for other, non-separated Lagrangian boundary value problems no such restrictions are valid. In contrast to corollary \ref{cor:dimrestr} we have

\begin{prop}\label{prop:nobifurrestrperiod}
In periodic boundary value problems for smooth families of symplectic maps on a $2n$-dimensional symplectic manifold all singularities up to degeneracy $2n$ occur.
Moreover, the generic singularities (i.e.\ non-removable under small perturbations) are naturally in 1-1 correspondence to the generic singularities occurring in the critical-points problem (remark \ref{rem:recapCatastropheTheory}) with up to $2n$ variables.
\cite[Prop. 3.4]{bifurHampaper}
\end{prop}

\section{Obstructions for singularities in systems with conformal symplectic symmetry}\label{sec:ObstrConfSym}

By corollary \ref{cor:dimrestr}, the geometric structure of separated Lagrangian boundary value problems forbids bifurcations whose fully reduced representatives of the corresponding critical-points problems need more variables than half the dimension of the phase space.
It turns out that in systems with a (non-trivial) conformal symplectic, $k$-dimensional group action leaving the boundary condition and the motions invariant up to time-rescaling even stronger restrictions apply.
Indeed, if the Hamiltonian vector field is not tangential to the boundary condition then in $2n$-dimensional systems only singularities of degeneracy at most $n-k$ occur in separated Lagrangian boundary value problems. Moreover, the requirement that the boundary condition is invariant can be weakened to the condition that the action is tangential to the boundary condition. 
%This refines the result for general Lagrangian problems with separated boundary conditions of proposition \ref{prop:dimrestr} stating that the maximal degree of degeneracy is $n$ (which is the lowest upper bound in a general setting).
The new result applies to homogeneous Hamiltonian systems and we will show that this provides an alternative viewpoint on geodesic bifurcations and recovers classical results about multiplicities of conjugate points along geodesics.

\subsection{General conformal symplectic symmetric Hamiltonians}

\begin{definition}[Conformal symplectic map]
Let $(M,\omega)$ and $( M', \omega')$ be two symplectic manifolds. A map $\chi \colon M \to  M'$ is called a \textit{conformal symplectic map} if $\chi^\ast  \omega' = \theta \cdot \omega$ for some $\theta \in \R$.
\end{definition}

\begin{example}\label{ex:thetachi} Let $N$ be a smooth manifold. The cotangent bundle $M=T^\ast N$ with the projection map $\pi \colon M \to N$ can be equipped with the standard symplectic structure $\omega = -\d \lambda$, where $\lambda$ denotes the Liouvillian-1-form on $M$ which is canonically defined by
\begin{equation}\label{eq:Liouvillian}
\lambda_\alpha(v) = \alpha(\d \pi|_{\alpha}(v)) \quad \text{for all }\, \alpha \in M,\, v \in T_\alpha M.
\end{equation}
To $\theta \in \R$ consider $\chi\colon M \to M$ with $\chi(\alpha)=\theta \cdot \alpha$, where $\cdot$ denotes the usual scalar multiplication of forms with real numbers. The map $\chi$ is a conformal symplectic map with  $\chi^\ast \omega = \theta \cdot \omega$ because for each $\beta \in M$ and $w \in T_\beta M$ we have
\[
(\chi^\ast \lambda)|_{\beta}(w)
= \lambda|_{\theta \cdot \beta}(\d \chi (w))
\stackrel{\eqref{eq:Liouvillian}}= (\theta \cdot \beta)( \d (\pi \circ \chi) (w))
= (\theta \cdot \beta) (\d \pi (w))
= (\theta \cdot \lambda)|_{\beta}(w),
\]
where we have used that $\chi$ preserves the fibres of the cotangent bundle.
\end{example}

\begin{remark} Let $(M,\omega)$ and $(M',\omega')$ be symplectic manifolds with $\dim M >2$.
If a map $\chi\colon M \to M'$ fulfils $\chi^\ast \omega' = \theta \cdot \omega$ for a smooth map $\theta \colon M \to \R$ then $\theta$ is constant, i.e.\ $\chi$ is a conformal symplectic map: since $\omega$ is closed,
\[
0 = \chi^\ast(\d \omega')
=\d \chi^\ast( \omega')
= \d (\theta \cdot \omega)
= \d \theta \wedge \omega + \theta \wedge \d \omega
= \d \theta \wedge \omega.
\]
By non-degeneracy of $\omega$ and $\dim M > 2$ it follows that $\theta$ is constant. Therefore, $\chi$ is a conformal symplectic map. \cite{TrGroupDiffGeo}
\end{remark}

Before formulating the main theorem of this section, we prove the following:

\begin{lemma}[Conformal symplectic transformations of Hamiltonians]\label{lem:homtime} Let $(M,\omega,H)$ and $( M',\omega', H')$ be two Hamiltonian systems with flow maps $\phi_t$ and $ \phi_t'$, respectively. Consider a conformal symplectic diffeomorphism $\chi \colon M \to  M'$ with $\chi^\ast  \omega' = \theta \cdot \omega$ for $\theta \in \R$. If $ H' \circ \chi = \eta \cdot H$ for $\eta \in \R$ then
\[
\chi \circ  \phi_{(\eta / \theta) \cdot t}=  \phi'_t \circ \chi
\]
(wherever defined).
\end{lemma}

\begin{proof}
Since $\chi$ is a diffeomorphism and $ \omega'$ is non-degenerate, $\chi^\ast  \omega' = \theta\cdot \omega$ is non-degenerate, so that $\theta \not =0$. In particular, the time-scaling factor ${\eta}/{\theta}$ in the assertion is well-defined.
Define Hamiltonian vector fields $X_H$ and $ X'_{ H'}$ by
%\[ \d H =  \omega(X_H, .), \qquad \d  H' =   \omega' ( X'_{ H'}, .).\]
\[
\d H =  \iota_{X_H}\omega, \qquad \d  H' =  \iota_{X'_{ H'}} \omega',
\]
where $\iota_X \omega$ denotes the contraction of the 2-form $\omega$ with a vector field $X$. We calculate
%\begin{align*}
%\theta \cdot \omega (\chi^\ast  X'_{ H'},.)
%&= (\chi^\ast \omega') (\chi^\ast  X'_{ H'},.)
%= \chi^\ast \big( \omega' (  X'_{ H'},.)\big)\\
%&= \chi^\ast \d  H'
%= \d ( H' \circ \chi)
%= \eta\cdot \d H 
%= \eta \cdot \omega(X_H,.).
%\end{align*}

\begin{align*}
\theta \cdot \iota_{\chi^\ast  (X'_{ H'})}\omega
&= \iota_{\chi^\ast  (X'_{ H'})}(\chi^\ast \omega')
= \chi^\ast \big( \iota_{X'_{ H'}}\omega'\big)\\
&= \chi^\ast \d  H'
= \d ( H' \circ \chi)
= \eta\cdot \d H 
= \eta \cdot \iota_{X_H}\omega.
\end{align*}

Since $\theta \not=0$ we have
%\[
%\omega (\chi^\ast  X'_{ H'},.)
%= \frac{\eta}{\theta} \omega(X_H,.) 
%=  \omega\left(  \frac{\eta}{\theta} X_H,.\right) 
%=  \omega\left( X_{({\eta}/{\theta}) \cdot H},.\right) .
%\]
\[
\iota_{\chi^\ast  (X'_{ H'})}\omega
= \frac{\eta}{\theta} \iota_{X_H}\omega
=  \iota_{({\eta}/{\theta}) X_H}\omega
=  \iota_{X_{({\eta}/{\theta}) \cdot H}}\omega.
\]
By non-degeneracy of the symplectic form $\omega$ it follows that the vector fields $X_{({\eta}/{\theta}) \cdot H}$ and $\chi^\ast  X'_{ H'}$ coincide. Notice further that if $\dot \gamma(t)=X_H(\gamma(t))$ then
\[
\frac{\d}{\d t} \big(\gamma(\alpha t)\big)
=\alpha \dot \gamma( \alpha t)
= \alpha X_H(\gamma( \alpha t))
= X_{\alpha H}(\gamma( \alpha t))
\]
for $\alpha \in \R$, i.e.\ $t \mapsto \gamma(\alpha t)$ is a flow line of the Hamiltonian vector field corresponding to the Hamiltonian $\alpha H$. We can conclude that the following statements are equivalent.
\begin{itemize}
\item
The curve $t \mapsto  \gamma(t) \in M$ is a flow line of $ X_{H}$ through $\gamma(0)$ at $t=0$. 
\item
The curve $t \mapsto  \gamma((\eta/\theta) \cdot t) \in M$ is a flow line of $ X_{(\eta/\theta) \cdot H} = \chi^\ast  X'_{ H'}$ through $\gamma(0)$ at $t=0$.  
\item
The curve $t \mapsto (\chi \circ \gamma)((\eta/\theta) \cdot t) \in  M'$ is a flow line of $ X'_{ H'}$ through $\chi(\gamma(0))$ at $t=0$. 
\end{itemize}
Thus
\[
\chi \circ  \phi_{({\eta}/{\theta})\cdot t}=\phi'_t \circ \chi.
\]
\end{proof}

\begin{example}\label{ex:HamInvuptoscaling} Consider the Hamiltonian
\[
H(\alpha) = \frac 12 \alpha(g^{-1}(\alpha)), \quad \alpha \in T^\ast N
\]
from \eqref{eq:HamGeoCordsfree} on the cotangent bundle $(M=T^\ast N,\omega)$ over a smooth Riemannian manifold $(N,g)$. %
The multiplicative Lie group of positive real numbers $\R^+ = (0,\infty)$ acts on $M$ by the conformal symplectic diffeomorphisms $\chi_\theta \colon M \to M$, $\chi_\theta(\alpha)=\theta \cdot \alpha$ analysed in example \ref{ex:thetachi}. We have $H \circ \chi_\theta = \theta^2 \cdot H$ such that for the Hamiltonian flow $\phi_t$
\begin{equation}\label{eq:timescalinggeodesics}
\chi_\theta \circ \phi_{\theta \cdot t} = \phi_{t} \circ \chi_{\theta}
\end{equation}
by lemma \ref{lem:homtime}. The motions of the Hamiltonian system $(M,\omega,H)$ correspond to the velocity vector fields of the geodesics on $N$ under the bundle isomorphism between $TN$ and $T^\ast N$ defined by $g$. Therefore \eqref{eq:timescalinggeodesics} corresponds to the fact that a geodesic $ \gamma'$ starting at $q \in N$ with initial velocity $\dot { \gamma'}(0) = \theta \cdot v \in T_q N$ reaches the same point $ \gamma'(t) \in N$ after time $t$ as a geodesic $\gamma$ starting at $q \in N$ with initial velocity $\dot \gamma(0) = v \in T_q N$ after time $\theta \cdot t$ and the end velocities fulfil $\theta \cdot \dot { \gamma'}(t) = \dot \gamma(\theta \cdot t)$.
\end{example}

Consider a Hamiltonian system $(\tilde M, \omega, H)$ with Hamiltonian flow map denoted by $\phi_t$ and two Lagrangian submanifolds $\Lambda, \Lambda' \subset \tilde M$. Denote the time-1-map $\phi_1$ by $\phi$ and assume that there exists $z \in \Lambda$ with $z'=\phi(z) \in \Lambda'$. Consider the separated Lagrangian boundary value problem $(\phi, \Lambda,\Lambda')$ localised around $z \in M\subset \tilde M$, $z'\in M' \subset \tilde M${, where $M$, $M'$ are open subsets of $\tilde M$}.

The following theorem shows that if $H$ is
invariant up to scaling under a conformal symplectic action of a $k$-dimensional Lie group
acting tangentially to $\Lambda$ at $z$ 
and $X_H(z') \not \in T_{z'}\Lambda'$ then the degeneracy of the singularity at $z$ is at most $\frac 12 \dim M-k$.
This refines the statement of corollary \ref{cor:dimrestr} and applies, for instance, everywhere away from $y =0$ to the Dirichlet-type problem \eqref{eq:DirichletProblem} if $H$ is homogeneous in $y$.

\begin{theorem}[Singularity obstructions for conformal symplectic Hamiltonian boundary value problems]\label{thm:confsymHam}
Consider the separated Lagrangian boundary value problem $(\phi, \Lambda,\Lambda')$ localised around a solution $z \in \Lambda\subset(M,\omega) $, $z' = \phi(z)\in \Lambda'\subset  (M',\omega)$ in an ambient Hamiltonian system $(\tilde M,\omega,H)$ with flow map $\phi_t$ denoting $\phi_1=\phi$.
Consider a Lie group $G$ with action $g \mapsto \chi_g$ on $M \cup M'$ defined locally on a neighbourhood $U$ of the neutral element $e \in G$ and the neighbourhood $\mathfrak U = \exp^{-1}(U)$ of $0$ in the Lie algebra of $G$.
Assume that there exist smooth maps $\theta, \eta \colon \mathfrak U \to \R$ such that for all $V\in \mathfrak U$
\begin{enumerate}
\item
the action is conformal symplectic, $\chi_{\exp(V)}^\ast \omega = \theta(V) \cdot \omega$,
\item
the Hamiltonian is invariant under the action up to scaling, $H \circ \chi_{\exp(V)} = \eta(V) \cdot H$,
\item
the time scaling factor for the flow map is not stationary at $s=0$, i.e.\
\begin{equation}\label{eq:nonstationaryassumption}
\left.\frac{\d}{\d s}\right|_{s=0}\left( \frac{\eta(sV)}{\theta(sV)}\right) \not =0,
\end{equation}
\item
the group acts tangentially to $\Lambda$ at $z$, i.e.\ $V^{\#}_z \in T_z \Lambda$, where $V^{\#}$ is the fundamental vector field corresponding to $V$,
\item
and the Hamiltonian vector field is not tangent to $\Lambda'$ at $z'$, i.e.\ $X_H(z') \not \in T_{z'}\Lambda'$.
\end{enumerate}
Then the degeneracy of a singularity of the Lagrangian boundary value problem at $z$ is at most $\frac 12 \dim M - \dim G$.
\end{theorem}
%, i.e.\
%\[
%\dim \left(\ker \d \left.\left[\phi|_\Lambda\right]\right|_z \right) \le \frac 12 \dim M-1.
%\]

%We can reformulate the assertion as follows: 
%
%\begin{corollary}
%In the boundary value problem considered in theorem \ref{thm:confsymHam} only those singularities occur which can be obtained in the gradient-zero problem $\nabla g_\mu=0$ for a smooth family of smooth maps $g_\mu \colon \R^k \to \R$ with $k \le \frac 12 \dim M-\dim G$.
%\end{corollary}

%\begin{remark}
%It is only slightly more restrictive to exchange bullet points 2-4 by the condition that
%\begin{itemize}
%\item
%there exists a Lie-group action of $\R^+=(0,\infty)$ on $M\cup M'$ by conformal symmetries. 
%\end{itemize}
%If a Lie-group $G$ with neutral element $e$ acts on $M\cup M'$ by conformal symplectic maps $\Xi_g$ then the 1-parameter family $(\chi_s)_{s \in I}$ might occur as $\chi_s=\Xi_{g(s)}$ for a path $g \colon I \to G$ with $g(0)=e$. 
%\end{remark}

We prepare the proof of theorem \ref{thm:confsymHam} by discussing the assumptions.

\begin{remark}

The assumptions of the theorem imply that $\theta(V) \not =0$ for all $V \in \mathfrak U$
and
\begin{equation}\label{eq:thetaeta1}
\theta(0)=1=\eta(0).
\end{equation}
The quotient ${\eta(sV)}/{\theta(sV)}$ is the time-scaling factor appearing in lemma \ref{lem:homtime}. Applied to the setting of the theorem, the lemma says
\begin{equation}\label{eq:lemmaapplied}
\chi_{\exp(sV)} \circ \phi_{{\eta(sV)}/{\theta(sV)} \cdot t} = \phi_t \circ \chi_{\exp(sV)}.
\end{equation}
\end{remark}

\begin{remark}\label{rem:foliation}
Let $\dim M=2n$ and $\dim G=k$.
The non-stationary assumption \eqref{eq:nonstationaryassumption} can be interpreted as a non-degeneracy assumption on the action. It implies that the intersections of the isotropy groups near $z$ and $z'$ with $U$ are trivial and that the group orbits of $G$ constitute a $k$-dimensional foliation around $z$ and $z'$. In particular, $k \le \frac 12 \dim M$ due to the tangency condition. 
\end{remark}

\begin{remark}
The non-stationary assumption \eqref{eq:nonstationaryassumption} ensures that the time-scaling in \eqref{eq:lemmaapplied} depends on $s$ to linear order. This means that the assumptions do \textit{not} apply to symplectic group actions leaving $H$ invariant. 
\end{remark}

%\begin{remark}
%As discussed in the introduction section to separated Lagrangian boundary value problems \ref{subsec:introDiri}, the assumption that the projection to the leave space $M' \to M'/\mathcal D$ is a submersion fixes a smooth manifold structure on $M'/\mathcal D$ compatible with $M'$. In exotic cases this can fail (see footnote on page \pageref{footnote:noindstr}).
%\end{remark}

%\begin{remark}
%If $(\tilde M, \tilde \omega,H)$ denotes the ambient Hamiltonian system, $G$ is a Lie-group with neutral element $e$ and
%\[
%\Xi \colon G \to \{ \Xi_g \colon \tilde M \to \tilde M \, | \, \Xi_g^\ast \tilde \omega = \kappa \cdot \tilde \omega, \kappa \in \R \setminus \{0\} \}
%\]
%is a Lie-group action on $M$ acting by conformal symplectic maps 
%then the 1-parameter family $(\chi_s)_{s \in I}$ might occur as $\chi_s=\Xi_{g(s)}$ for a path $g \colon I \to G$ with $g(0)=e$. 
%\end{remark}

\begin{proof}[Proof of theorem \ref{thm:confsymHam}]
\textit{Step 1.} We construct an integrable distribution $\mathcal D$ over $M'$ with leaves consisting of orbits and with $\mathcal D_{z'}=T_{z'}\Lambda'$.

Let $2n=\dim M'$ and $k = \dim G$.
By remark \ref{rem:foliation} we have $k \le n$ and the orbits of the group action provide a $k$-dimensional foliation $O$ of $M'$ (shrinking $M'$ around $z'$ if necessary). There exists an $n-k$-dimensional manifold $N$ containing $z'$ which is transversal to each orbit $O_a$ with $a \in N$ such that $T_{z'}N \oplus T_{z'} O_{z'} = T_{z'}\Lambda'$. The collection of orbits $L_{z'}=\bigcup_{a \in N}O_a$ defines an $n$-dimensional submanifold of $M'$. By construction $T_{z'}L_{z'} = T_{z'}\Lambda'$.

We extend $L_{z'}$ to a foliation of $M'$ as follows: shrinking $M'$, if necessary, the projection map $\pi \colon M' \to M'/  O$ to the $2n-k$-dimensional space of leaves is a submersion and $\pi(L_{z'})$ is a smooth, $n-k$-dimensional submanifold of $M'/  O$ (thought of as $N$). Shrinking $M'$, if necessary, we find a foliation $\mathcal L$ of $M'/  O$ with leaf $\mathcal L_{\pi(z')} = \pi(L_{z'})$. The preimages of the leaves of $\mathcal L$ under $\pi^{-1}$ form the desired foliation of $M'$ giving rise to an integrable distribution $\mathcal D$.

\textit{Step 2.} We show that the kernel of the map
%there is an element in $T_z \Lambda$ which is not mapped to $0 \in T_{[z']}(M' / \mathcal D)$ by the map
\[ 
\d \left.\left[\phi_{|\Lambda} \right]\right|_z \colon T_z \Lambda \to T_{[z']} \left(M'/\mathcal D\right)\]
is at {most $n-k$}-dimensional such that the claim follows by proposition \ref{prop:kernelDegeneracy}.
Let $V \in \mathfrak U$ with fundamental vector field $V^{\#}$. By assumption $V^{\#}_z \in T_z \Lambda$. We calculate
\begin{align*}
\d \left.\left[\phi_{|\Lambda} \right]\right|_z \left(V^{\#}_z\right)
&=\left.\frac{\d}{\d s}\right|_{s=0}
\left[\left({\phi_1} \circ \chi_{\exp(sV)}\right)(z) \right]\\
&\stackrel{\eqref{eq:lemmaapplied}}=\left.\frac{\d}{\d s}\right|_{s=0}
\left[\left( \chi_{\exp(sV)} \circ {\phi_{\eta(sV)/\theta(sV) \cdot 1}}\right)(z) \right]\\
&\stackrel{(\ast)}=\left.\frac{\d}{\d s}\right|_{s=0}
\left[\left(  {\phi_{\eta(sV)/\theta(sV) \cdot 1}}\right)(z) \right]\\
&=\left[\left.\frac{\d}{\d s}\right|_{s=0} \left(\frac{\eta(sV)}{\theta(sV)}\right) \cdot X_H\left(\phi_{\eta(0)/\theta(0) }(z)\right) \right]\\
&\stackrel{\eqref{eq:thetaeta1}}=\left[\underbrace{\left.\frac{\d}{\d s}\right|_{s=0} \left(\frac{\eta(sV)}{\theta(sV)}\right)}_{\not=0 \text{ by } \eqref{eq:nonstationaryassumption}} \cdot \underbrace{X_H(z')}_{\not \in \mathcal D_{z'}=T_{z'}\Lambda'} \right]\\
&\not =0.
\end{align*}
The equality $(\ast)$ is due to the invariance of the distribution $\mathcal D$ under the group action.
\end{proof}

Theorem \ref{thm:confsymHam} applies to Hamiltonian systems where the Hamiltonian is homogeneous in the direction of $\Lambda$ and $\Lambda'$.
%on polarised manifolds where the polarisation contains the tangent bundle over $\Lambda$ and $\Lambda'$ and the Hamiltonian is homogeneous in the direction of the polarisation.
Consider the cotangent bundle over a smooth manifold with symplectic coordinates $(x,y) = (x^1,\ldots,x^n,y_1,\ldots,y_n)$ and a homogeneous Hamiltonian in the $y$-coordinates. Then theorem \ref{thm:confsymHam} applies to the Dirichlet-type problem \eqref{eq:DirichletProblem} leading to the following

%As before, we consider Lagrangian Dirichlet boundary value problems formulated as in \eqref{eq:DirichletProblem}. The following proposition and corollary reveal how homogeneity of the Hamiltonian induces obstructions upon which singularities can occur in Lagrangian Dirichlet boundary value problems in Hamiltonian systems.

\begin{prop}[Singularity obstruction for separated Lagrangian boundary value problem in homogeneous Hamiltonian systems]\label{prop:DimResthomDirichlet}Consider a Hamiltonian system $(\R^{2n},\omega_{\mathrm{std}},H)$ with $H(x,\lambda y) = \lambda^p H(x,y)$ with $p \not = 1$ and $\nabla_yH(x,y) \not=0$ for $y \not = 0$. Then
\[
\dim \ker D_y \phi^X \le n-1,
\]
at points which are not mapped to the ${y=0}$-subspace by the Hamiltonian time-1-map $\phi$.
\end{prop}

\begin{remark}
The condition $\nabla_yH(x,y) \not=0$ for $y \not = 0$ is equivalent to the condition that for any fixed $x$ the map $y \mapsto H(x,y)$ is not constant on a 1-dimensional subspace of the $y$-coordinate plane. %It makes sure that the homogeneity has a non-trivial effect.
For instance, the condition $H(x,y) \not =0$ for $y\not=0$ implies $\nabla_yH(x,y) \not=0$ for $y \not = 0$ for homogeneous Hamiltonians in $y$.
\end{remark}

\begin{proof}[Proof of proposition \ref{prop:DimResthomDirichlet}]
We consider the positive real numbers $\R^+$ as a multiplicative Lie group with action $\chi_r(x,y)=(x,  r y)$, $\theta(r)=r$ and $\eta(r)=r^p$. The action preserves the fibres $\{x\} \times \R^n$ of the cotangent bundle $T^\ast \R^n \cong \R^{2n}$.
The assumption $p \ge 2$ implies that the time-scaling factor $\eta(r)/\theta(r) = r^{p-1}$ is not stationary at the neutral element $1$.
%, i.e.\
%\[\left.\frac{\d}{\d s}\right|_{s=0}\frac{\eta(\exp(s))}{\theta(\exp(s))} = (k-1) \not =0.\]
Due to $\nabla_y H(x,y) \not =0$ for $y \not =0$ the Hamiltonian vector field $X_H$ is not tangent to any fibre $\{x\}\times \R^n$ unless $y=0$. Now the assertion follows by theorem \ref{thm:confsymHam}.
\end{proof}

Instead of referring to theorem \ref{thm:confsymHam}, we can also carry out the proof of proposition \ref{prop:DimResthomDirichlet} directly in local coordinates. It corresponds to choosing the vertical polarisation
\[\mathcal D = \mathrm{span} \left\{ \frac{\p}{\p y_1},\ldots,\frac{\p}{\p y_n}\right\},\]
in the theorem's proof.

\begin{proof}[Direct proof of proposition \ref{prop:DimResthomDirichlet}]
Let $(a,b) \in \R^n\times\R^n \cong \R^{2n}$ with $b \not=0$. The map $y\mapsto \phi^X(a,y)$ has an $n$-dimensional domain space. We prove the claim by showing that $b \in \R^n \cong T_{b}\R^n$ is not element of the kernel $\ker D_y \phi^X(a,b)$; the assertion then follows by proposition \ref{prop:kernelDIM}.
\begin{align*}
\D \phi^X(a,b) \begin{pmatrix}
0\\b
\end{pmatrix}
%&=\lim_{\e \to 0} \frac 1 \e \left( \phi^X(a,(1+\e)b) - \phi^X(a,b) \right)\\
&=\lim_{\e \to 0} \frac 1 \e \left( (x\circ\phi_1)(a,(1+\e)b) - (x\circ\phi_1)(a,b) \right)\\
&\stackrel{(\ast)}{=}\lim_{\e \to 0} \frac 1 \e \left( (x\circ\phi_{(1+\e)^{p-1}})(a,b) - (x\circ\phi_1)(a,b) \right)\\
&=\left.\frac{\d}{\d t}\right|_{t=1}(x\circ\phi_{t^{p-1}})(a,b) \\
&= {(p-1)}\d x(X_H(\phi_1(a,b)))\\
&= {(p-1)}\d x(X_H((x\circ\phi)(a,b),(y\circ\phi)(a,b)))\\
&={(p-1)}\nabla_yH ((x\circ\phi)(a,b),(y\circ\phi)(a,b))\\
&\not=0,
\end{align*}
where we have used lemma \ref{lem:homtime} to obtain the equality $(\ast)$.% and that $\d x \circ X_H$ is the $(\p_{x^1},\ldots \p_{x^n})$-coordinate of the vector field $X_H$.
\end{proof}

\begin{example}
Proposition \ref{prop:DimResthomDirichlet} applies to Hamiltonian systems with Hamiltonians of the form
\begin{equation}\label{eq:HamM}
H(x,y) = \frac 12 y^T A(x)y \quad \text{with} \quad A(x) \in \mathrm{GL}(n,\R) \text{ for all } x,
\end{equation}
where $\mathrm{GL}(n,\R)$ denotes the group of invertible matrices. (The restriction to invertible matrices forces $\nabla_y H (x,y)=0 \implies y=0$ as required in proposition \ref{prop:DimResthomDirichlet}.) 
\end{example}

\begin{remark}
Recall that by the Darboux-Weinstein theorem \cite[Thm.\ 7.1]{WEINSTEIN1971329} polarisations transversal to a Lagrangian submanifold $\Lambda$ give rise to symplectomorphisms from local neighbourhoods of $\Lambda$ to neighbourhoods of the zero section in the cotangent bundle $T^\ast \Lambda$ {such that} the polarisation is carried to the vertical polarisation in $T^\ast \Lambda$. The canonical coordinates on the cotangent bundle $T^\ast \Lambda$ can be used to induce local symplectic coordinates on neighbourhoods of the submanifold $\Lambda$ in the ambient symplectic manifold. This shows that the statement of \ref{prop:DimResthomDirichlet} applies to Hamiltonians which are homogeneous in the direction of any polarisation tangential to the boundary condition.
\end{remark}

\begin{example}\label{ex:multiscalingsym}
To linearly independent $a^{(s)}=(a^{(s)}_1,\ldots,a^{(s)}_n) \in \R^n$ with $1\le s \le k \le n$ consider the action of $(\R^+)^k=(0,\infty)^k$ on $(\R^{2n},\omega_{\mathrm{std}},H)$ by the conformal symplectic transformations $\chi^{(s)}_\lambda$ with
\begin{align*}
x^j &\mapsto \lambda^{a^{(s)}_j} x^j\\
y_j &\mapsto \lambda^{c-a^{(s)}_j} y_j.
\end{align*}
Assume that $H\circ  \chi^{(s)}_\lambda = \lambda^p \cdot H$ for all $s \le k$. Without loss of generality, $c=1$. Assume that $p\not=1$. % to obtain a non-trivial time scaling of the flow map when changing coordinates with $\chi_\lambda^{(s)}$.
Since scaling symmetries commute, the Lie group action is well-defined.
Moreover, the action fulfils the non-stationary assumption \eqref{eq:nonstationaryassumption} since $p\not=1$ and the $a^{(s)}$ are linearly independent.

If $H$ is a polynomial and separated, i.e.\ its Hessian is of block-diagonal form, then $H$ is of the form
\[
H(x,y) = 
\sum_{b=1}^{m_1} \alpha_b \prod_{j=1}^{l_b^{(1)}} x^{\sigma^{(1)}_b(j)}
+\sum_{b=1}^{m_2} \beta^b \prod_{j=1}^{l_b^{(2)}} y_{\sigma^{(2)}_b(j)}
\]
with $m_1,m_2\in \N_0$, $l_b^{(1)},l_b^{(2)}\in \N_{>0}$, $\{\alpha_b\}_1^{m_1} ,\{\beta_b\}_1^{m_2} \subset \R\setminus\{0\}$ and maps
\[
\sigma^{(\tau)}_{b_\tau} \colon \{1,\ldots,l_{b_\tau}^{(\tau)}\} \to \{1,\ldots,n\} \quad \text{for } 1\le {b_\tau}\le m_\tau, \; \tau \in \{1,2\}
\]
such that for all $1\le s \le k$
\[
\forall 1\le b\le m_1 :
\sum_{j=1}^{l_b^{(1)}}a^{(s)}_{\sigma^{(1)}_b(j)} = p
\quad \text{and}\quad
\forall 1\le b\le m_2 :
\sum_{j=1}^{l_b^{(2)}}(1-a^{(s)}_{\sigma^{(2)}_b(j)})=p.
\]
Let $\Lambda$ be a Lagrangian submanifold which is tangent to the vector field
\[
\sum_{j=1}^n a^{(s)}_j x^j \frac{\p}{\p x^j} + (1-a^{(s)}_j)y_j \frac{\p}{\p y_j}
\]
at $z$ for all $s \le k$ and let $\Lambda'$ be any Lagrangian manifold transversal to the Hamiltonian vector field $X_H$.
By theorem \ref{thm:confsymHam}, the maximal degeneracy of a singularity occurring in the Lagrangian boundary value problem $(\phi,\Lambda,\Lambda')$ for the time-1-map $\phi$ of the Hamiltonian flow at $z$ is $n-k$. 
\end{example}

The following reformulation of theorem \ref{thm:confsymHam} can be handy when considering conformal actions which naturally decompose into a conformal and symplectic action.

\begin{corollary}\label{corr:confsymHamDecompose}
Let $(\phi, \Lambda,\Lambda')$ be a separated Lagrangian boundary value problem localised around a solution $z \in \Lambda\subset(M,\omega) $, $z' = \phi(z)\in \Lambda'\subset  (M',\omega')$ in an ambient Hamiltonian system $(\tilde M,\omega,H)$ with time-1-map $\phi$ such that $X_H(z') \not \in T_{z'}\Lambda'$.
Let $G$ be a Lie group acting conformally on $M$ and $M'$ by $g \mapsto \chi_g$ and symplectically by $g \mapsto \Psi_g$, where the actions are defined locally around a neighbourhood $U$ of the neutral element $e \in G$ and the neighbourhood $\mathfrak U = \exp^{-1}(U)$ of $0$ in the Lie algebra of $G$. Moreover, assume that for all $V\in \mathfrak U$ the fundamental vector fields $V^\#$ and $V^{\#\#}$ corresponding to the infinitesimal actions with $V$ fulfil
\[
V_z^\# + V_z^{\#\#} \in T_z \Lambda.
\]
Define $\theta \colon \mathfrak U \to \R$ by $\chi_{\exp(V)}^\ast \omega = \theta(V) \cdot \omega$ and assume there exists $\eta \colon \mathfrak U \to \R$ fulfilling the non-stationary assumption \eqref{eq:nonstationaryassumption} such that for all $V\in \mathfrak U$
\[ H \circ \chi_{\exp(V)} \circ \Psi_{\exp(V)} = \eta(V) \cdot H.\]
Then the degeneracy of a singularity of the Lagrangian boundary value problem at $z$ is at most $\frac 12 \dim M - \dim G$.
\end{corollary}

%\begin{corollary}
%Consider a Lagrangian Dirichlet problem $(\phi, \Lambda,\Lambda')$ in a Hamiltonian system. Assume that a transversal Lagrangian distributions exists such that the Hamiltonian is homogeneous of degree at least 2 in the momentum coordinates induced by the Darboux-Weinstein theorem. Moreover, let the Hamiltonian be non-zero for non-vanishing values of the momentum variables.
%Only those singularities occur which can be obtained in the gradient-zero problem $\nabla g_\mu=0$ for a smooth family of smooth maps $g_\mu \colon \R^k \to \R$ with $k \le n-1$. 
%\end{corollary}

\begin{example}
The scaling symmetries in example \ref{ex:multiscalingsym} naturally decompose into symplectic scaling transformations $\Psi_\lambda^{(a)}$ with 
\begin{align*}
x^j &\mapsto \lambda^{a_j} x^j\\
y_j &\mapsto \lambda^{-a_j} y_j
\end{align*}
and the conformal symplectic transformation $\chi_\lambda(x,y)=(x,\lambda y)$ considered in example \ref{ex:thetachi}.
\end{example}

\subsection{Geodesic bifurcations}\label{subsub:applgeod}

\subsubsection{Application of theorem \ref{thm:confsymHam}}
Let us use the results to analyse which singularities occur in conjugate loci on smooth Riemannian manifolds. As mentioned in the introductory example in section \ref{sec:introductionsection}, the motions of the Hamiltonian system $(T^\ast N,\omega,H)$ with
\begin{equation}\label{eq:HamiltonianGeodesic}
H(\alpha) = \frac 12 \alpha\left(g^{-1}(\alpha)\right)
\end{equation}
and standard symplectic structure $\omega$ correspond to the velocity vector fields of geodesics on the smooth Riemannian manifold $(N,g)$. Finding a geodesic which connects two points $x,X\in N$ corresponds to solving the boundary value problem 
\begin{equation}\label{eq:bdproblHGeodesicx}
\pi (\alpha) =x, \quad (\pi \circ \Phi)(\alpha)=X,
\end{equation}
where $\Phi$ is the time-1-map of the Hamiltonian flow on $T^\ast N$. 
Together with the observations in example \ref{ex:HamInvuptoscaling}, theorem \ref{thm:confsymHam} reproves the following classical fact \cite[Ch.5]{doCarmo}.
\begin{prop}\label{prop:kernelExpMap}
The {dimension of the} kernel of the geodesic exponential map on a Riemannian manifold {is strictly less than the dimension of the manifold}.
\end{prop}

As well as applying theorem \ref{thm:confsymHam} directly to conclude proposition \ref{prop:kernelExpMap}, one can also use proposition \ref{prop:DimResthomDirichlet} because $H$ takes the form \eqref{eq:HamM} when expressed in canonical coordinates where $A$ is the matrix representation of $g^{-1}$ in the corresponding frame.

Recall that if $\gamma$ is a geodesic starting at time $t=0$ at $x$ with $\dot \gamma(0) = y$ then $x$ and $\gamma(1)$ are conjugate points if and only if the kernel of the exponential map evaluated at $y \in T_xN$ is non-trivial. The dimension of the kernel corresponds to the multiplicity of the conjugate points.
%This is equivalent to the introduction via Jacobi-fields recalled in the introduction \ref{sec:introductionsection}.
Thus, proposition \ref{prop:kernelExpMap} implies

%number of linearly independent Jacobi fields\footnote{vector fields along a geodesic arising as variational vector fields for variations through geodesics} along $\gamma$ which vanish at $x$ and $\gamma(1)$. If there exists a non-trivial Jacobi vector field then the points $x$ and $\gamma(1)$ are called \textit{conjugate points} and the number of linearly independent Jacobi vector fields vanishing at $x$ and $\gamma(1)$ is called the \textit{multiplicity} of $x$. Now the following statements can be interpreted as corollaries.

\begin{prop}\label{cor:multconjpts}
The multiplicity of conjugate points on an $n$-dimensional Riemannian manifold $N$ cannot exceed $n-1$. {In other words, at points in a conjugate locus with respect to some fixed point in $N$ the degeneracy of a singularity is at most $n-1$.}
\end{prop}

\begin{remark}
The space of Jacobi fields along a geodesic $\gamma$ vanishing at $\gamma(0)$ corresponds to the space $T_{\gamma(0)}N$ via $J \mapsto \frac{\nabla}{\d t} J(0)$, where $\frac{\nabla}{\d t}$ denotes the covariant derivative along $\gamma$ with respect to the Levi Civita connection on $(N,g)$. Classically, corollary \ref{cor:multconjpts} is proved by noting that the Jacobi field $t \mapsto t \dot \gamma(t)$ along the curve $\gamma$ does not vanish at $\gamma(1)$ \cite[Ch.5]{doCarmo}.
\end{remark}

In the introductory example in section \ref{sec:introductionsection} we saw isolated cusp singularities connected by lines of fold singularities in the conjugate locus of a 2-dimensional Gaussian and in the conjugate locus of a 2-dimensional ellipsoid embedded in $\R^3$ (figure \ref{fig:EllipsoidConjPts}). The singularities persist under small perturbations of the {metric or the reference point}. The occurrence of fold and cusp singularities as generic singularities in conjugate loci on surfaces is not a coincidence.
{Consider the conjugate locus to a fixed point $x$ on an $n$-dimensional Riemannian manifold $(N,g)$.
The Hamiltonian boundary value problem \eqref{eq:bdproblHGeodesicx} has $n$ parameters given by the position of the endpoint $X$.}

%Under non-degeneracy assumptions on the metric, by the correspondence of Hamiltonian boundary value problems with catastrophe theory (observation \ref{obs:transl}) we expect at most those singularities to occur generically which occur generically in the critical-points problem for families with $n$ parameters in at most $n-1$ variables.

Assuming that considering the position of the endpoint $X$ as parameters of the boundary value problem  \eqref{eq:bdproblHGeodesicx} yields a versal unfolding of the corresponding critical-points-of-a-function problem (observation \ref{obs:transl}) and that no other obstructions than the one described by Theorem \ref{thm:confsymHam} apply (cf. \cite[p.362]{Janeczko1995}), those singularities occur generically, i.e. unremovably under small perturbations of the metric or the reference point, which occur generically in the critical-points problem for families with $n$ parameters in at most $n-1$ variables.
By the classification results in \cite{Arnold2012} this means that in conjugate loci on Riemannian surfaces fold ($A_2$) and cusp ($A_3$) singularities occur generically. Cusp singularities occur at isolated points and fold singularities in 1-parameter families.

If we allow the reference point $x$ to vary as well on the manifold, then even more singularities can become generic phenomena because the corresponding Hamiltonian boundary value problem has $2n$ parameters, namely the reference point $x$ and the endpoint $X$. However, the degeneracy of the singularities is still bounded by $n-1$. 
Again, assuming that considering $x$ and $X$ as parameters of the corresponding Hamiltonian boundary value problem \eqref{eq:bdproblHGeodesicx} yields a versal unfolding of the problem, those singularities occur generically which also occur generically in the critical-points problem for families with $2n$ parameters in at most $n-1$ variables. On a Riemannian surface these are fold ($A_2$), cusp ($A_3$), swallow tail ($A_4$) and butterfly ($A_5$) singularities, where the singularities $A_j$ occur generically in $5-j$-parameter families.

This recovers results from \cite{WATERS20171} where the creation and annihilation of cusps in the conjugate locus on a surface is analysed by studying Jacobi's equations.

\subsubsection{Further remarks on bifurcations in the conjugate-point problem}\label{subsubsec:furtherRemarks}
There is a variety of aspects to the conjugate-points problem on a Riemannian manifold.
{Fixing the reference point, the exponential map can be viewed as a Lagrangian map \cite{Janeczko1995} for which generic singularities are classified \cite[Part III]{Arnold2012}. Moreover, under non-degeneracy assumptions the conjugate locus of an exponential map with respect to a fixed reference point can locally be identified with families of skew-symmetric matrices \cite{CutLocusWeinstein}.}
Illustrative and related to the presented viewpoint is reference \cite{WATERS20171} which includes an analysis how the conjugate locus as a whole can bifurcate as the reference point moves on a surface.
A functional analytic approach to the geodesic bifurcation problem (in an extended sense) can be found in \cite{Piccione2004}.
Calculating geodesics on submanifolds of the euclidean space is often motivated by the task of finding distance minimising curves between two points. Several methods are presented in \cite{GeodesicCourseWork}. An approach using geodesics as homotopy curves can be found in \cite{Thielhelm2015}. Moreover, the authors examine implications of the symplectic structure in the problem for numerical computations in \cite{numericalPaper}: {a correctly captured elliptic umbilic bifurcation ($D_4^{-}$) in the conjugate locus of a perturbed 3-dimensional ellipsoid can be seen in figure \ref{fig:RATTLEElliptic}. Three lines of cusp bifurcations ($A_3$) merge in an elliptic umbilic point.
If the symplectic structure was ignored in the discretisation of the problem, the lines of cusps would fail to merge and no elliptic umbilic point would be visible.}

\begin{figure}\begin{center}
\includegraphics[width=0.6\textwidth]{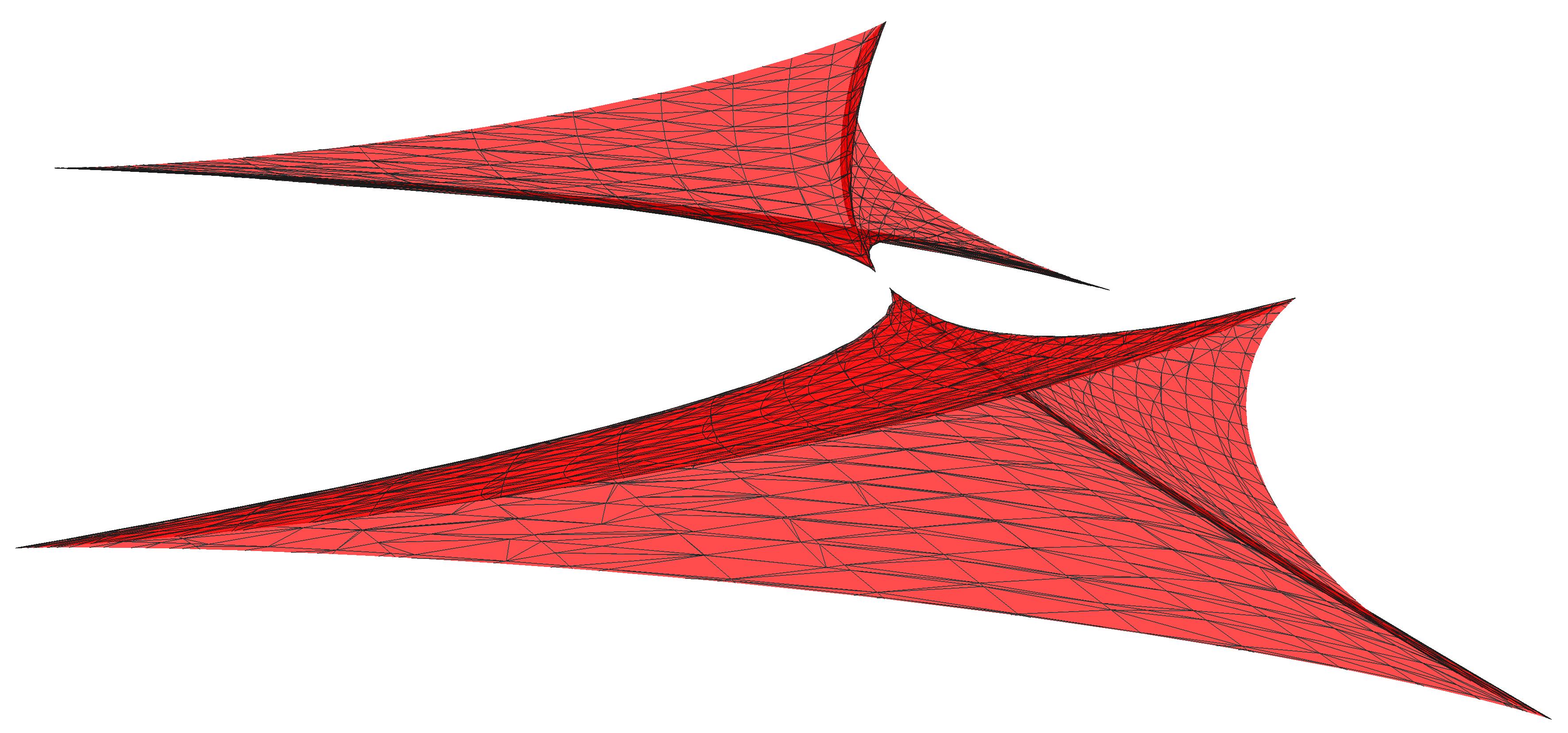}
\end{center}
\caption{Elliptic umbilic bifurcation $D_4^-$ in the conjugate-points-problem on a perturbed 3-dimensional ellipsoid embedded in $\R^4$. Three lines of cusps merge to an elliptic umbilic singularity. To capture this bifurcation numerically, preserving the symplectic structure of the problem when discretising is essential. (See \cite{numericalPaper} for details.)}\label{fig:RATTLEElliptic}
\end{figure}

%% file: acknowledgements.tex
\section*{\qquad \qquad \qquad \qquad \qquad \quad \; Acknowledgements}

We thank Peter Donelan, Bernd Krauskopf, Hinke Osinga and Gemma Mason for useful discussions.
This research was supported by the Marsden Fund of the Royal Society Te Ap\={a}rangi.